\documentclass[letterpaper,12pt]{article}

\usepackage{fullpage}
\usepackage{enumerate}
\usepackage{authblk}
\usepackage[colorinlistoftodos]{todonotes}
\usepackage{verbatim}
\usepackage{section, amsthm, textcase, setspace, amssymb, lineno, amsmath, amssymb, amsfonts, latexsym, fancyhdr, longtable, ulem}
\usepackage{mathrsfs} 
\usepackage{epsfig, graphicx, pstricks,pst-grad,pst-text,tikz, tkz-berge,colortbl}

 \usepackage{float}

\usepackage{mathtools}

\usetikzlibrary{graphs,graphs.standard,positioning,decorations.pathreplacing}

\newtheorem{theorem}{Theorem}
\newtheorem{lemma}[theorem]{Lemma}

\newcommand{\R}{\mathbb R}

\begin{document}

\author{Vincent E. Coll, Jr. and Lee B. Whitt}

\title{The Flat Plane\\ and a\\Constructive Proof of Minding's Theorem}

\maketitle

\bigskip \date\!\!{} {\noindent \textit{Department of Mathematics, Lehigh University, Bethlehem, PA 18015\\ Northrop Grumman - Senior Technical Fellow \rm{(}$retired$\rm{)}, $San~Diego, CA~92014$}

\section{Introduction} 

The simplest example of a Riemannian manifold is the Cartesian plane with the standard Euclidean (flat) plane, which we here denote by $g_0$. From the Cartan-Hadamard theorem, we know that the Euclidean metric is unique in that every complete flat metric on $\R^2$ is isometric to $g_0$. By restricting $g_0$ to proper open sets in the plane, one obtains an inexhaustible supply of incomplete metrics. These are examples of what we call \textit{subset metrics}.  That is, metrics obtained by the restriction of the Euclidean metric to a subset of the plane or more generally by an isometric embedding to a subset.  For example, subset metrics can also be constructed on the entire plane as a product metric like  
$e^{-2x^2} dx^2 + e^{-2y^2} dy^2$.  In this case, we have  a subset metric on $\R^2$ which can be ``realized" as an isometric embedding into $\R^2$ as an open square with sides of length $\pi$.  However, the plane admits interesting flat metrics which are both natural and incomplete -- but which are decidedly not subset metrics.  

\section{Minding's theorem - a constructive proof}

Ferdinand Minding (1806-1885) was an influential German-Russian mathematician best known for his contributions to the differential geometry of surfaces of constant curvature.  Minding's results on the geometry of geodesic triangles on a surface of constant curvature (1840) anticipated Beltrami's approach to the foundations of non-Euclidean geometry (1868) and his work in this area brilliantly continued the investigations of Gauss who had pioneered the intrinsic notion of curvature in his remarkable Theorem Egregium (1828).  

Minding's most celebrated result is his namesake theorem of 1839 which established that all surfaces having the same constant curvature must be locally isometric. Today,  Minding's theorem is a staple in differential geometry textbooks. But, to the best of our knowledge, all published proofs of it, inclusive of Minding's original argument\footnote{Unfortunately, Minding's paper is only available in the original German \textbf{\cite{Minding}} or in a Russian translation.}, are existential in nature.  In this section, we give a constructive proof of Minding's theorem in the flat case.  The proof requires only some basic facts about harmonic functions and complex analytic functions.

\bigskip
Two metrics $g_0$ and $g_{\varphi}$ on a 
surface $S$ are \textit{conformally equivalent} if $g_{\varphi}=e^{2\varphi} g_0$ where $\varphi$ is a smooth real-valued function on $S$. 
The function $e^{2\varphi}$ is called the \textit{conformal factor}.  The curvature $K_0$ of the initial metric $g_0$ and the curvature $K_{\varphi}$ of the conformally changed metric $g_{\varphi}$, are related by the well-known equation involving $\Delta_0 \varphi$, the Laplacian on $(S,g_0)$:

\begin{eqnarray}\label{KW}
K_{\varphi}=e^{2\varphi}(-\Delta_{0} \varphi + K_0).
\end{eqnarray}

\noindent
Here, we are concerned with $S=\R^2$ and we will take $g_0$ to be the standard flat metric on $\R^2$.  Evidently, $K_0\equiv 0$ and it is an easy exercise to show that $K_\varphi$ is also identically zero.
Equation (\ref{KW}) then implies that $\Delta_0 \varphi=\Delta \varphi$ (the standard Laplacian) is zero.  Maps with vanishing Laplacian are called \textit{harmonic} and appear in profusion.\footnote{Note that equation (2) establishes that the set of harmonic functions is unaltered by a conformal change of metric.} The real and imaginary parts of any complex analytic function are harmonic.  And, as we will see in the following lemma, if $\varphi$ is a harmonic function on $\R^2$ and $\nabla \varphi$ is its gradient, then $\log (||\nabla \varphi||_0)$ is also harmonic.
(Here, $||~~||_0$ refers to length defined by the metric $g_0$.)

\begin{lemma} \label{firstlemma}

If $\mathscr{U}\subseteq \R^2$ is a simply connected open set and $F:\mathscr{U} \rightarrow \R$ is a harmonic function, then there exists a simply connected open set $\mathscr{V}\subseteq \mathscr{U}$ and a harmonic function $G:\mathscr{V}\rightarrow\R$ such that

$$F= G + \log(\| \nabla G \|) . $$
\end{lemma}
\begin{proof} Let $f$ be a complex analytic function on $\mathscr{U}$ and $F=Re(f)$. If we can find a complex analytic function $g$ on 
$\mathscr{V}$ such that

\begin{equation} \label{complexversion}
f(z)=g(z) + \log (g^{\prime}(z)), ~~z=(x,y)\in \mathscr{V},
\end{equation}
then taking the real parts of the functions in equation
\eqref{complexversion} gives

$$
F(x,y) = G(x,y) + \log | g^{\prime}(z) |, \rm{~~where~~ 
\textit{G}=Re(\textit{g})}.
$$
But since $g$ is complex analytic, there exists a harmonic function $H$ such that 
$$ g(z) = G(x,y) + i H(x,y). $$
We can write
$$ g^{\prime}(z) = \frac{\partial G}{\partial x} + i \frac{\partial H}{\partial x}. $$
and by\textit{ the Cauchy-Riemann equations} (``CR"), we have
$$ g^{\prime}(z) = \frac{\partial G}{\partial x} - i \frac{\partial G}{\partial y},  $$
which implies that
$$ | g^{\prime}(z) | = \left[ \left(\frac{\partial G}{\partial x}\right)^2 + \left(\frac{\partial G}{\partial y}\right)^2 \right]^{1/2} = \| \nabla G \|_0. $$
It remains to solve (\ref{complexversion}) for $g(z)$.  Elementary manipulations yield

\begin{eqnarray}\label{equation for g}
g(z) = \log \left( \int_{z_0}^z {\rm e}^{f(\xi)} d \xi + C \right), ~~z_0\in \mathscr{U}.  
\end{eqnarray}
Since $\displaystyle{{\rm e}^f}$ is analytic, the integral in \eqref{equation for g} is well-defined -- but it may be unbounded.  By shrinking $\mathscr{U}$ somewhat and considering a smaller open and simply connected $\mathscr{V}\subseteq \mathscr{U}$, we can pick $C$ so that $g$ is not zero (so that log is defined) in a neighborhood of $z_0\in \mathscr{V}$.
\end{proof}

We are now in a position to constructively prove the following special case of Minding's theorem, 
which constructively establishes that 
$(\R^2,g_\varphi)$ and $(\R^2,g_0)$ are locally isometric.  

\begin{theorem}\label{W}
If $p\in (\R^2,g_0)$, then there exists a simply connected open neighborhood $\mathscr{V}$ of $p$
and an isometry 
$W: \mathscr{V} \rightarrow (\mathbb{R}^2,g_\varphi) $ given by
$$ W(x,y) = \left( {\rm e}^{\Phi} \cos \Psi, {\rm e}^{\Phi} \sin \Psi \right), $$
where $\Phi$ is the harmonic solution on $\mathscr{V}$ to
$$ 
\Phi + \log \| \nabla \Phi \|_0 = \varphi, \rm{~~by~~ Lemma~~ \ref{firstlemma}}, 
$$
and $\Psi$ is a harmonic conjugate to $\Phi$.
\end{theorem}

\begin{proof} 
Let $<,>$ denote the standard inner product used to define the metric $g_0$:  

\begin{eqnarray*}\label{metric}
g_0=\left( 
  	\begin{array}{cc}
  		\left< \frac{\partial}{\partial x},\frac{\partial}{\partial x}\right> & \left< \frac{\partial}{\partial x},\frac{\partial}{\partial y}\right>  \\      
  		\left< \frac{\partial}{\partial y},\frac{\partial}{\partial x}\right>& \left< \frac{\partial}{\partial y},\frac{\partial}{\partial y}\right>  \\    
  	 \end{array}  
  \right)
=\left( 
  	\begin{array}{cc}
  		1 & 0  \\      
  		0 & 1  \\    
  	 \end{array}  
  \right).
\end{eqnarray*}

\noindent

To establish that $W$ is an isometry we need only show that the matrix of inner products that defines 
$g_\varphi$, namely

$$
g_\varphi=e^{2\varphi}\left( 
  	\begin{array}{cc}
  		\left< \frac{\partial}{\partial x},\frac{\partial}{\partial x}\right> & \left< \frac{\partial}{\partial x},\frac{\partial}{\partial y}\right>  \\      
  		\left< \frac{\partial}{\partial y},\frac{\partial}{\partial x}\right>& \left< \frac{\partial}{\partial y},\frac{\partial}{\partial y}\right>  \\    
  	 \end{array}  
  \right)
=\left( 
  	\begin{array}{cc}
  		e^{2\varphi} & 0  \\      
  		0 & e^{2\varphi}  \\    
  	 \end{array}  
  \right),
$$
is given by $W_*$, the standard push-forward map defined by the total derivative of $W$.  That is, we must show 

$$
\left\langle W_{\ast} \left( \frac{\partial}{\partial x} \right), W_{\ast} \left( \frac{\partial}{\partial x} \right) \right\rangle   = {\rm e}^{2 \varphi}=\left\langle W_{\ast} \left( \frac{\partial}{\partial x} \right), W_{\ast} \left( \frac{\partial}{\partial x} \right) \right\rangle, 
$$
and that 
$$
\left\langle W_{\ast} \left( \frac{\partial}{\partial x} \right), W_{\ast} \left( \frac{\partial}{\partial y} \right) \right\rangle =0.
$$

Computing, we have
\begin{eqnarray*}
W_{\ast} \left( \frac{\partial}{\partial x} \right) &=&
\left( {\rm e}^{\Phi} \Phi_x \cos \Psi - {\rm e}^{\Phi} \sin \Psi \Psi_x, {\rm e}^{\Phi} \Phi_x \sin \Psi + {\rm e}^{\Phi} \cos \Psi \Psi_x \right)                                               \\
&=& {\rm e}^{\Phi} \left( \Phi_x \cos \Psi + \Phi_y \sin \Psi, \Phi_x \sin \Psi - \Phi_y \cos \Psi \right), \mbox{ using CR.}
\end{eqnarray*}
So,
\begin{align*}
\left\langle W_{\ast} \left( \frac{\partial}{\partial x} \right), W_{\ast} \left( \frac{\partial}{\partial x} \right) \right\rangle  &=
{\rm e}^{2\Phi} \bigl[ \Phi_x^2 \cos^2 \Psi + 2 \Phi_x \Phi_y \cos \Psi \sin \Psi + \Phi_y^2 \sin^2 \Psi  \\
                      & \qquad  + \Phi_x^2 \sin^2 \Phi - 2 \Phi_x \Phi_y \sin \Psi \cos \Psi + \Phi_y^2 \cos^2 \Psi \bigr] \\
                      & = {\rm e}^{2 \Phi} \left[ \Phi_x^2 + \Phi_y^2 \right] \\
                      & = {\rm e}^{2 \Phi} \| \nabla \Phi \|_0^2 \\
                      & = {\rm e}^{2 \varphi}, \mbox{ by the construction from Lemma \ref{firstlemma}. }
\end{align*}
A similar calculation gives
$$ \left\langle W_{\ast} \left( \frac{\partial}{\partial y} \right), W_{\ast} \left( \frac{\partial}{\partial y} \right) \right\rangle =
{\rm e}^{2 \varphi}. $$
Finally,
\begin{align*}
\left\langle W_{\ast} \left( \frac{\partial}{\partial x} \right), W_{\ast} \left( \frac{\partial}{\partial y} \right) \right\rangle  &=
{\rm e}^{2 \Phi} [ ( \Phi_x \cos \Psi + \Phi_y \sin \Psi) (\Phi_y \cos \Psi - \Phi_x \sin \Psi ) \\
                      &  \qquad + ( \Phi_x \sin \Psi - \Phi_y \cos \Psi ) ( \Phi_y \sin \Psi + \Phi_x \cos \Psi ) ] \\
                      & = {\rm e}^{2 \Phi} [ \Phi_x \Phi_y \cos^2 \Psi - \Phi_x^2 \cos \Psi \sin \Psi + \Phi_y^2 \sin \Psi \cos \Psi \\
                      & \qquad - \Phi_y \Phi_x \sin^2 \Psi + \Phi_x \Phi_y \sin^2 \Psi + \Phi_x^2 \sin \Psi \cos \Psi \\
                      & \qquad - \Phi_y^2 \cos \Psi \sin \Psi - \Phi_y \Phi_x \cos^2 \Psi ] \\
                      & = 0.
\end{align*}
This completes the proof of the theorem.
\end{proof}

\noindent
\textit{Remark:}  This proof does not extend to dimensions greater than two.  
Even so, it is true that a flat metric on $\mathbb{R}^n$ is locally isometric to the standard metric $\mathbb{R}^n$.  

\bigskip
\noindent
\textbf{Example 1:}  Although Theorem \ref{W} provides an explicit isometry $W$ between  $(\R^2,g_0)$ and $(\R^2,g_\varphi)$
the integral in \eqref{equation for g} is generally not tractable -- unless $f$ is rather simple.  Consider the identity function  $f(z)=z=x+iy$.  Define $g_\varphi=e^{2\varphi}g_0$, where $\varphi= Re(f)=x$. In this case, $W(x,y)=(e^x \cos{y}, e^x \sin{y}). $

\end{document}